\documentclass[oneside,reqno]{amsart}
\usepackage[english]{babel}
\usepackage{indentfirst}
\usepackage[T1]{fontenc}
\usepackage[utf8]{inputenc}
\usepackage{paralist}
\usepackage{amssymb}
\usepackage{amsthm}
\usepackage{amsmath}
\usepackage{amscd}
\usepackage{graphicx}
\usepackage[colorlinks=true]{hyperref}
\usepackage[margin=1.5in]{geometry}

\theoremstyle{plain} 
\newtheorem{teo}{Theorem}[section] 
 
\newtheorem{lem}[teo]{Lemma}
\newtheorem{prop}[teo]{Proposition}
\newtheorem{defn}[teo]{Definition}
\newtheorem{oss}[teo]{Remark}
\numberwithin{equation}{section} 

\newcommand{\R}{\ensuremath{\mathbb{R}}}
\newcommand{\N}{\ensuremath{\mathbb{N}}}

\newcommand{\La}{\ensuremath{\mathcal{L}}}
\newcommand{\D}{\ensuremath{\bold{D}}}
\newcommand{\al}{\ensuremath{ &\;}}
\newcommand{\C}{\ensuremath{\mathbb{C}}}
\newcommand{\lr}[1]{\left( #1 \right)}
\newcommand{\lrq}[1]{\left[ #1 \right]}

  \newcommand{\pd}[2][]{\frac{\partial #2}{\partial #1}}
  \newcommand{\dpd}[2][]{\frac{\partial^2 #2}{\partial #1^2}}
\newcommand{\eqlab}[1]{\begin{equation}  \begin{aligned}#1 \end{aligned}\end{equation}}
\newcommand{\bgs}[1]{\begin{equation*} \begin{aligned}#1\end{aligned}\end{equation*}}
 \newcommand{\syslab}[2] []  {\begin{equation}#1  \left\{\begin{aligned}#2\end{aligned}\right.\end{equation}}
  \newcommand{\sys}[2][]{\begin{equation*}#1  \left\{\begin{aligned}#2\end{aligned}\right.\end{equation*}}
 \renewcommand{\phi}{\ensuremath{\varphi}}

\begin{document}
\nocite*
\date{}
\author{Claudia Bucur}
\author{Fausto Ferrari}
\address{Claudia Bucur: Dipartimento di Matematica\\ Universit\`a degli Studi di Milano \\ Via Cesare Saldini, 50 \\ 20100, Milano-Italy}
\address{Fausto Ferrari: Dipartimento di Matematica\\ Universit\`a di Bologna\\ Piazza di Porta S.Donato 5\\ 40126, Bologna-Italy}
\email{claudia.bucur@unimi.it}
\email{fausto.ferrari@unibo.it }
\title[An extension problem for the fractional derivative defined by Marchaud]{An extension problem for the  fractional derivative defined by Marchaud}
\begin{abstract}
We prove that the (nonlocal) Marchaud fractional derivative in $\R$ can be obtained from a parabolic extension problem with an extra (positive) variable, as the operator that maps the heat conduction equation to the Neumann condition.
Some properties of the fractional derivative are deduced from those of the local operator. In particular we prove a Harnack principle for Marchaud-stationary functions. 
\end{abstract}

\maketitle

\tableofcontents

\section{Introduction}

In literature there are several definitions of fractional derivatives (see, for instance, the monographs  \cite{samkokilbas}, \cite{Samko} and \cite{Butzer_Westphal} for an historical introduction). In particular, we are interested in the notion given by Marchaud, see \cite{Marchaud}, who introduced two types of fractional derivatives. For a fixed $s\in (0,1)$, the left and the right Marchaud fractional derivative of order $s $, see \cite{samkokilbas}, formulas 5.57 and 5.58, are respectively defined as follows:
\eqlab{ \label{mdefct}
{\D}^s_{\pm} f(t)=\frac{s}{\Gamma(1-s)}\int_{0}^{\infty}\frac{f(t)-f(t\mp \tau)}{\tau^{1+s}}d\tau.
}
 These fractional derivatives are well defined when $f$ is a bounded, locally H{ö}lder continuous function in $\R.$  In particular, we may assume that  $f\in C^{\bar{\gamma}}(\R),$ for $s<\bar{\gamma} \leq 1$ and $f\in L^{\infty}(\mathbb{R})$ (see the Appendix for further details), even though these hypotheses can be weakened. 
    In addition, we just recall here that  the Marchaud derivative can be defined for $s\in (0,n)$ and $n \in \N$, as
\[\D_{\pm}^s f(t) = \frac{\{s\}}{\Gamma(1-\{s\})} \int_{0}^{\infty} \frac{f^{[s]} (t)-f^{[s]}(t\mp\tau)}{\tau^{1+\{s\}}}\, d \tau,\]
where $[s]$ and $\{s\}$ denote, respectively, the integer and the fractional part of $s$. Our work focuses on the case $n=1$ and, in the first part of the paper, on the left fractional derivative, that we can write by a change of variable, neglecting the constant and omitting for simplicity the subscript symbol $+$, as:
	\begin{equation}\label{frader}
	\begin{split} \D^s f (t):= \al \int_{0}^\infty \frac{f(t)-f(t-\tau)}{\tau^{s+1}}\, d\tau= \int_{-\infty}^t \frac{f(t)-f(\tau)}{(t-\tau)^{s+1}}\, d\tau.
	 \end{split}
	 \end{equation}

A short remark on  the right counterpart of the Marchaud fractional derivative is given in Section \ref{rightderiv}. Moreover, we point out that we argue considering (\ref{frader}) as the definition of our fractional derivative without taking care of what happens when $s\to 0^+$ or $s\to 1^-$. Nevertheless, in the Appendix, we briefly discuss these interesting  cases with respect to the definition given in (\ref{mdefct}).

 The purpose of the present work is to introduce an extension operator for the fractional derivative introduced in \eqref{frader}. Indeed, the operator  $\D^s$ naturally arises when dealing with a weighted parabolic differential equation (the heat conduction problem) on the positive half-plane, with a positive space variable and for all times, namely for $(x,t)\in [0,\infty)\times \R $. 
 
 In order to construct this extension operator, we exploit the idea recently revisited in \cite{CAFSIL}. In that paper, the fractional Laplacian was characterized via an extension procedure, by means of a weighted second order elliptic local operator. 
 
Considering the function $\phi$ of one variable, formally representing the time variable, our approach relies on constructing a weighted parabolic local operator by adding an extra variable, say the space variable, on the positive half-line, and working on the extended plane $ [0,\infty)\times \R $.
 
 The heuristic argument can be described as follows. Let $\varphi:\mathbb{R}\to \mathbb{R}$ be a given function, sufficiently smooth. Let $U$ be a solution of the problem
 \syslab{ \label{pro1} & \pd[t]{U}=\dpd[x]{U},  &&(x,t) \in (0,\infty)\times\mathbb{R}\\
 		&U(0,t)=\varphi(t), &&t \in \R. }
 		
Let us point out that this is not the usual Cauchy problem for the heat operator, but a heat conduction problem. 

It is known that, without extra assumptions, we can not expect to have a unique solution of the problem (\ref{pro1}), see \cite{tichonov}. Nevertheless, if we denote by $T_{1/2}$ the operator that associates to $\varphi$ the partial derivative $ \displaystyle\pd[x]{ U},$ whenever $U$ is sufficiently regular, we have that
$$
T_{1/2}T_{1/2}\varphi=\frac{d\varphi}{dt}.
$$
That is $T_{1/2}$ acts like an half derivative, indeed 
$$
\frac{\partial}{\partial x}\frac{\partial U}{\partial x}(x,t)=\frac{\partial U}{\partial t}(x,t) \underset{x \rightarrow 0} \longrightarrow  \frac{d\varphi(t)}{dt}.
$$
The solution of problem (\ref{pro1}) under the reasonable assumptions that $\varphi$ is bounded and H{ö}lder continuous, is explicitly known (check e.g. \cite{tichonov}) to be
\bgs{ U(x,t)=\al c {x} \int_{-\infty}^t \displaystyle e^{-\frac{x^2}{4(t-\tau)}}{(t-\tau)^{-\frac{3}{2}}}\varphi(\tau)\, d\tau\\
=\al c {x}  \int_{0}^{\infty} \displaystyle e^{-\frac{x^2}{4\tau}}{\tau^{-\frac{3}{2}}}\varphi (t-\tau)\, d\tau,
} 
where the last line is obtained with a change of variable. Using $\displaystyle t =\frac{x^2}{4\tau}$ and the integral definition of the Gamma function (see formula 6.1.1 in \cite{ABRAMOWITZ}) we have that 
\[ \int_0^\infty x e^{-\frac{x^2}{4\tau}} \tau ^{-\frac{3}2} \, d\tau = 2 \int_0^\infty e^{-t} t^{-\frac{1}2}\, dt = 2 \Gamma\lr{\frac{1}2}.\]
Hence,
\bgs{ \frac{U(x,t)-U(0,t)}{x}=\al c\int_{0}^{\infty}e^{-\frac{x^2}{4\tau}}{\tau^{-\frac{3}{2}}}\left(\phi(t-\tau)-\phi(t)\right)d\tau,}
choosing $c$ that takes into account the right normalization. This yields, by passing to the limit, that
\begin{equation}
-\lim_{x\to 0^+}\frac{U(x,t)-U(0,t)}{x}=c\int_{0}^{\infty}\frac{\phi(t)-\phi(t-\tau)}{\tau^{\frac{3}{2}}}d\tau.
\end{equation}
Hence, with the right choice of the constant, we get exactly $\D^{1/2}\phi$ i.e. the Marchaud derivative  of order $1/2$ of $\phi.$

Now we are in position to state our main result.
 
\begin{teo} \label{teo:mainstat}
Let $s\in (0,1) $ and $\varphi \colon \R \to \R$ be a bounded, locally $C^{\bar{\gamma}}$ function for $s<\bar{\gamma}\leq1$. Let $U\colon [0,\infty)\times \R\to \R$ be 
a solution of the problem 
\syslab{ \label{prob1}
	& \pd[t]{ U (x,t)} = \frac{1-2s}x \pd[x]{ U (x,t)}+ \dpd[x]{ U(x,t)},  & &   (x,t)\in(0,\infty)\times \R\\
	 &U(0,t)=\varphi(t),  &&t\in\R\\
	  & \lim_{x \to \infty} U(x,t)=0. && 
}
Then $U$ defines the extension operator for $\phi$, such that
\eqlab{\label{mainstat} \D^s \varphi(t)=-\lim_{x\to 0^+} c_s x^{-2s}(U(x,t)- \varphi(t)) ,}
where \[ c_s= 4^s\Gamma(s).\] \end{teo}

We notice that one can write
	\eqlab{ \label{mainstat1} \D^s \varphi(t)= -\lim_{x\to 0^+} c_sx^{1-2s}\pd[x]{ U}(x,t),}
	in analogy with formula (3.1) in \cite{CAFSIL}.



\begin{oss}
The extension defined in \eqref{mainstat1} satisfies, as one would expect, up to constants:
 	\[ \D^{1-s} \D^{s}  \varphi(t)= \varphi'(t).\]
Indeed, 
\begin{equation}
\begin{split}
\D^{1-s} \D^s \varphi(t) =&  \lim_{x\to 0^+} x^{2s-1} \pd[x]{} \left(x^{1-2s} \pd[x]{ U}(x,t)\right)\\=  & \lim_{x\to 0^+} \dpd[x]{U}(x,t) + \frac{1-2s}x \pd[x]{ U}(x,t)\\
=&\lim_{x\to 0^+} \pd[t]{ U}(x,t)=\pd[t]{ U}(0,t)=\varphi'(t) .
\end{split}
\end{equation}
	\end{oss}
	
An interesting application that follows from this extension procedure is a Harnack inequality for Marchaud-stationary functions in an interval $J \subseteq {\mathbb{R}},$ namely for functions that satisfy  $\D^s \phi=0$  in  $J.$ This result is not trivial, since the fractional stationary functions on an interval of $\mathbb{R}$ determine a nontrivial set of functions, see e.g. \cite{Bucur}.  
\begin{teo}\label{teoHarn}
Let $s\in (0,1)$. There exists a positive constant $\gamma$ such that, if $\D^s\phi=0$ in $J\subseteq \R$ and $\varphi\geq 0$ in $\mathbb{R}$, then  
\begin{equation}\label{Harnack_intro}
\sup_{[t_0-\frac{3}{4}\delta,t_0-\frac{1}{4}\delta]}\phi\leq \gamma \inf_{[t_0+\frac{3}{4}\delta,t_0+\delta]}\phi
\end{equation}
for every $t_0\in \R$ and for every  $\delta >0$ such that $[t_0-\delta,t_0+\delta]\subset J$.
\end{teo}

This result can be deduced from the Harnack inequality proved in \cite{CS} for some weighted parabolic operators. In particular, the constant $\gamma$ used in the previous Theorem \ref{teoHarn} is the same that appears in the parabolic 
Harnack case, see \cite{CS}. 
In addition, we remark that the inequality (\ref{Harnack_intro}) is not the usual Harnack inequality for elliptic operators, where the comparison between the supremum and the infimum is done on the same set, e.g. the same metric ball.
This Harnack inequality for the Marchaud-stationary functions inherits the behavior of its parabolic extension. 

\section{The weighted parabolic problem} 

In this section we find a solution of the system in \eqref{prob1}. At first, we introduce a particular kernel, that acts as the Poisson kernel. We then look for a particular solution of the system by means of the Laplace transform, and in this way we show how the solution arises. Finally, by a straightforward check, it yields that indeed the indicated solution satisfies the problem \eqref{prob1}.

 \subsection{Properties of the kernel $\Psi_s$}
In this section we introduce and study the properties of a kernel, that acts as the Poisson kernel for the problem \eqref{prob1}.

We define  for every $x\in \mathbb{R},$
\syslab[\Psi_s(x,t) :=] { \label{kern} & \frac{1}{4^s \Gamma(s)} x^{2s} e^{-\frac{x^2}{4t}} t^{-s-1},& \mbox{ if }& t> 0 ,\\
								&0, & \mbox{ if }& t\leq 0 .} 
	Also, let
	\syslab[\psi_s(t) :=] {  \label{kern1} & \frac{1}{4^s \Gamma(s)} e^{-\frac{1}{4t}} t^{-s-1},& \mbox{ if }& t> 0 ,\\
								&0, & \mbox{ if }& t\leq 0 } 
		and notice that
	\begin{equation} \label {bla4} \int_\R \Psi_s(x,t)\, dt= \int_\R \psi_s(t)\, dt.
	\end{equation}
	Indeed, we have by changing coordinates $\tau=\displaystyle \frac{t}{x^2}$ 
	\begin{equation*}
	\begin{split}  \int_{\R} \Psi_s(x,t)\, dt  =\al \frac{1}{4^s\Gamma(s)} \int_0^\infty x^{2s}e^{-\frac{x^2}{4t}} t^{-s-1} \, dt\\ 
	=\al\frac{1}{4^s\Gamma(s)} \int_0^\infty e^{-\frac{1}{4\tau}} \tau^{-s-1} \, d\tau\\
	= \al\int_\R \psi_s(t)\, dt.
	\end{split}
	\end{equation*}
	
The kernel $\Psi_s$ satisfies also the following property:

\eqlab{\label{kencalc1} \int_{\R} \Psi_s(x,t)\, dt = 1.}
Indeed  by recalling \eqref{bla4} and performing the change of variables $ t = \displaystyle \frac{1}{4\tau},$ we get
	\begin{equation} \label{bla2}
	\begin{split}
	\int_\R \psi_s(\tau)\, d\tau = \frac{1}{4^s\Gamma(s)}   \int_0^\infty e^{-\frac{1}{4\tau}} \tau^{-s-1} \, d\tau
	 =  \frac{1}{\Gamma(s)} \int_0^\infty e^{-\tau} \tau^{s-1}\, d \tau .
	 \end{split}
	 \end{equation}
	 From the integral definition of the Gamma function (see formula 6.1.1 in \cite{ABRAMOWITZ})
	 \[ \Gamma(s) =\int_0^\infty e^{-t} t^{s-1}\, d t,\]
	 it follows that
				  \bgs{  \int_{\R} \Psi_s(x,t)\, dt  =1.  }  		

Taking the Laplace transform
 of the kernel $\Psi_s$, we have the following result involving the modified Bessel function of the second kind $\bold K_s,$ see \cite{MaGoberSo} and
 \cite{ABRAMOWITZ}, \textsection 9.6.
 \begin{lem}The Laplace transform of the function $\psi_s \in L^1(\R)$ is
\eqlab{ \label{fourfpsi} \La (\psi_s)(\omega)= \frac{1}{2^{s-1}\Gamma(s)} \omega^{\frac{s}2}\bold K_s(\sqrt{\omega}) \mbox{ for } \Re \omega >0.}
Moreover, the Laplace transform with respect to the variable $t$ of the kernel $\Psi_s \in L^1(\R, dt)$ is
\eqlab{ \label{calcKpsi} \La (\Psi_s)(x,\omega)=\frac{1}{2^{s-1}\Gamma(s)} x^s \omega^{\frac{s}2}\bold K_s(x\sqrt{\omega}) \mbox{ for } \Re \omega >0. }
\end{lem}	
\begin{proof}
If one proves claim \eqref{fourfpsi}, the second result \eqref{calcKpsi} follows after a change of variables. We have that
 \bgs{ \La (\Psi_s)(x,\omega) = \frac{1}{4^s\Gamma(s)}\int_0^\infty x^{2s} e^{-\frac{x^2}{4t} } t^{-s-1} e^{-\omega t} \, dt.} Taking $t=x^2\tau$ (and recalling that $x>0$), we obtain
 \bgs{ \La  (\Psi_s)(x,\omega) = \al \frac {1}{4^s\Gamma(s)}\int_0^\infty  e^{-\frac{1}{4\tau} } \tau^{-s-1} e^{-\omega (x^2 \tau)} \, d\tau\\
 =\al \La  (\psi_s)(x^2 \omega).} 
For $\Re a>0$ and  $\omega\in \C$ with $\Re\omega >0$, as stated in formula 5.34 in \cite{oberl}, we have that
 	\[\La \lr{t^{\gamma-1} e^{-\frac{a}{t} }} = 2 \lr{\frac{a}{\omega}}^{\frac{\gamma}2} \bold K_\gamma \lr{ 2(a\omega)^{\frac{1}2}}  .\] 
 Taking $\gamma=-s$, recalling that $\bold K_s=\bold K_{-s}$, we obtain
 \[ \int_0^\infty e^{-\frac{1}{4\tau}}  \tau^{-s-1} e^{-\omega \tau}\, d\tau= 2^{s+1} {\omega}^{\frac{s}{2}}\bold K_s(\sqrt{ \omega}).\] 
Hence, multiplying by  $4^s\Gamma(s)$ we get:
	\[ \frac{1}{4^s\Gamma(s)} \int_0^\infty e^{-\frac{1}{4\tau}}  \tau^{-s-1} e^{-\omega \tau}\, d\tau = \frac{1}{2^{s-1}\Gamma (s) } {\omega}^{\frac{s}{2}}\bold K_s(\sqrt{ \omega})\] and therefore \eqref{fourfpsi}.
\end{proof}			
   
\subsection{Existence of the solution}
\label{exuniq}
We prove in this section the following existence theorem:
\begin{teo}\label{teo:sol}
There exists a
continuous solution of the problem \eqref{prob1} 
given by
\eqlab{ \label{solsts1} U(x,t)	=   \Psi_s(x,\cdot)*\varphi (t) := \int_{\R} \Psi_s(x,\tau)\varphi (t-\tau)\, d\tau.}
More precisely (inserting the definition \eqref{kern}) we have that
	\eqlab{\label{solsts2} U(x,t)=\frac{1}{4^s \Gamma(s)} x^{2s} \int_0^{\infty} e^{-\frac{x^2}{4\tau}} \tau^{-s-1}  \varphi(t-\tau) \, d\tau.}
\end{teo}

 Before proving this theorem, we recall a useful result (see \cite{FerFra}, Proposition 4.1) involving the modified Bessel function of the second kind. 
\begin{prop}{\label{prop:FerFra}} 
If $-\infty <\alpha<1$, the boundary value problem 
	\syslab{ \label{proby1}
		& x^{\alpha} y''(x) =y(x) &\mbox{in } &(0,\infty)\times \R\\
		& y(0)=1  & &\\
		& \lim_{x \to \infty} y(x) =0. &&}
		has a solution $\psi \in C^{2-\alpha} \lr{[0,\infty)}$ of the form
	\eqlab{\label{soly1} 
			\psi(x)= c_{k} x^{\frac{1}2}\bold{K}_{\frac{1}{2k}} \lr{\frac{t^k}k},} where 
			\[ c_k = \frac{2^{1-\frac{1}{2k}} k^{-\frac{1}{2k}}}{\Gamma \lr{\frac{1}{2k}}} \] is a positive constant and
				\[ k := \frac{2-\alpha}{2}.\]
				\end{prop}

We continue by showing how the solution of problem \eqref{prob1} arises, by means of the Laplace transform (see \cite{dyke} for details on this integral transform).
 \begin{proof}[Proof of Theorem \ref{teo:sol}]
 
 First, we look for a possible candidate of a solution of the problem \eqref{prob1}. In particular,  we put ourselves in the simplified situation that $U$ has a sub-exponential growth in $t$ and that the function $\varphi$ is zero on the negative semi-axis $(-\infty,0]$.
 Under this additional hypothesis, we take the Laplace transform in $t$ of the system \eqref{prob1}. Since the derivative of the Laplace transform acts as 
\bgs{ \La (f')(\omega) =  \omega \mathcal{L}f(\omega),} we get that
\sys{
	& \omega \mathcal{L} U(x,\omega) = \frac{1-2s}x \pd[x]{ \mathcal{L} U}(x,\omega) + \dpd[x]{ \mathcal{L} U}(x,\omega), 	&\mbox{in } &(0,\infty)\times \R\\
	&\mathcal{L} U(0,t)=\mathcal{L} \varphi(\omega),  &\mbox{in } &\R\\
	&\lim_{x \to \infty} \mathcal{L}U(x,\omega)=0, &\mbox{in } &\R.}
We define 
\eqlab{ \label{deff} f(x):= \mathcal{L} U(x,\omega),}
 then $f $ must be solution of the system
 	\syslab{ \label{probf1}
 	&  \omega  f(x)= \frac{1-2s}x  f'(x) +  f''(x),	 &\mbox{in } &(0,\infty)\\
	&f(0)=\mathcal{L} \varphi(\omega) & &\\
	&\lim_{x \to \infty} f(x)=0.&&}
	
	We consider the problem \eqref{proby1} for $\alpha=\frac{2s-1}{s}$ (notice that for $s\in(0, 1),$ we indeed have that $\alpha\in (-\infty,1)$)
	\sys{ \label{proby11}
		& x^{\frac{2s-1}s} y''(x) =y(x) &\mbox{in } &(0,\infty)\times \R\\
		& y(0)=1  & &\\
		& \lim_{x \to \infty} y(x) =0. &&}
		We observe that by taking 
		\eqlab{\label{chvarfy} f(x) = \mathcal{L} \varphi (\omega) y\lr{ \omega^s \lr{\frac{x}{2s}}^{2s}}}
		one obtains that $f$ satisfies problem \eqref{probf1}.
		Indeed, we have that
		\[ y'\lr{ \omega^s \lr{\frac{x}{2s}}^{2s}} = \frac{1}{\mathcal{L} \varphi(\omega) }\omega^{-s} (2s)^{2s-1} x^{1-2s} f' (x),\]
		and 
		\bgs{ y''\lr{ \omega^s \lr{\frac{x}{2s}}^{2s}} = \al f''(x) \frac{1}{\mathcal{L} \varphi(\omega)} \omega^{-2s} (2s)^{4s-2} x^{2-4s} \\ \al + f'(x)\frac{1-2s}{\mathcal{L} \varphi (\omega)}(2s)^{4s-2}\omega^{-2s}  x^{1-4s} .}
		Then, since by \eqref{proby1} we have 
			\[ \lr{ \omega^s \lr{\frac{x}{2s}}^{2s}} ^{\frac{2s-1}s} y''\lr{ \omega^s \lr{\frac{x}{2s}}^{2s}} =y\lr{ \omega^s \lr{\frac{x}{2s}}^{2s}} ,\] this implies that 
			\[ f(x)=\omega^{-1} \lr{f''(x) +(1-2s) x^{-1} f'(x)},\] which is the first equation in \eqref{probf1}. The fact that $f(0)= \mathcal{L} \varphi(\omega)$  easily follows.

A solution of problem \eqref{proby1} is given explicitly in Proposition \ref{prop:FerFra}. For $\alpha=\frac{2s-1}{s}$, we have $k=\frac{1}{2s}$ and
	\eqlab{\label{soly2} y(x)= \frac{2^{1-s}(2s)^s}{\Gamma(s)}x^{\frac{1}{2}} \bold{K}_s\lr{2s x^{\frac{1}{2s}}}.}
	Inserting this into \eqref{chvarfy}, we obtain that 
	\bgs{ f(x)= \mathcal{L} \varphi(\omega) \frac{2^{1-s}}{\Gamma(s)} \omega^{\frac{s}2} {x^s} \bold{K}_s(x\sqrt{\omega} ) .}
	Hence, by \eqref{deff}, we have that
	\[ \mathcal{L} U(x,\omega) = \mathcal{L}\varphi(\omega) \frac{2^{1-s}}{\Gamma(s)} \omega^{\frac{s}2} {x^s} \bold{K}_s(x\sqrt{\omega} ).\]
	By taking the inverse Laplace transform, recalling that the pointwise product is taken into the convolution product, we obtain that
		\eqlab{\label{bla1}  U(x,t)= \al \frac{2^{1-s}}{\Gamma(s)} \varphi * \La^{-1} \lr{ \omega^{\frac{s}2}{x^s}  \bold{K}_s(x\sqrt{\omega} ) }(t).}
			From \eqref{calcKpsi} we have that 
		\bgs{ \La (\Psi_s)(x,\omega)= \frac{\Gamma(s) }{2^{1-s}}  x^s \omega^{\frac{s}2} \bold K_s \lr{x \sqrt{\omega}}.}
			Hence in \eqref{bla1} we obtain the following Laplace convolution
			\eqlab {\label{bla3} U(x,t) = \phi*\Psi_s(x,t) = \int_0^t \Psi(x,\tau) \phi(t-\tau).}

We recall now that we took the function $\phi$ to be vanishing for $t\in (-\infty,0)$. 
Hence, it is reasonable to suppose that the above formula holds true also for a function that is not a signal. Hence, now  we assume that $\varphi$ is defined on the entire axis  $\R$,  without the assumption that $\varphi$ vanishes in $(-\infty,0).$ We claim that $\phi*\Psi_s$ still defines a solution of the problem \eqref{prob1} as in \eqref{solsts1}.
Indeed, we show now that  the function $U$ defined in \eqref{solsts1} solves the system \eqref{prob1}.

In definition \eqref{solsts1}, we call
	\sys[A_{x,\tau}:= ] { \al e^{-\frac{x^2}{4\tau}}\tau^{-s-1}, &\mbox{if } &\tau>0\\
	\al 0&\mbox{if } &\tau\leq 0 }
	and notice that
		\sys[\frac{ \partial A_{x,\tau}}{\partial x}= ] { \al  -\frac{x}{2\tau}A_{x,\tau}, &\mbox{if } &\tau>0\\
	\al 0&\mbox{if } &\tau\leq 0.}
	Let 
		\bgs{ V(x,t):=4^s\Gamma(s) U(x,t) =  x^{2s}\int_\R A_{x,\tau}\varphi(t-\tau)\, d\tau.}
Taking the derivative with respect to $x$ of $V(x,t)$ we have that
	\bgs{ \pd[x]{ V} (x,t)= 2s x^{2s-1} \int_\R A_{x,\tau} \phi(t-\tau) \, d\tau -\frac{x^{2s+1}}2  \int_\R \frac{A_{x,\tau}}{\tau} \phi(t-\tau) \, d\tau,}
	and
	\bgs{ \dpd[x]{V}(x,t)=\al  2s(2s-1) x^{2s-2}  \int_\R A_{x,\tau} \phi(t-\tau) \, d\tau - \frac{(4s+1)x^{2s}}{2}  \int_\R \frac{A_{x,\tau}}{\tau} \phi(t-\tau) \, d\tau\\ \al  + \frac{x^{2s+2}}4  \int_\R \frac{A_{x,\tau}}{\tau^2} \phi(t-\tau) \, d\tau.} 
Then, by changing variables, we write
	\[ V(x,t)=x^{2s} \int_\R A_{x,t-\tau} \varphi(\tau) \, d\tau ,\] and taking the derivative with respect to $t$ we have 
	\[\pd[t]{ V}(x,t)= x^{2s} \int_\R  \lrq{x^2 \frac{A_{x,t-\tau}}{4(t-\tau)^2 } \varphi(\tau) -(s+1) \frac{A_{x,t-\tau}}{(t-\tau) }\varphi(\tau) } \, d\tau.\]
	We change back variables to obtain	
	\[\pd[t]{ V}(x,t)  =- x^{2s+2}\int_\R  \frac{A_{x,\tau}}{4\tau^2 } \varphi(t-\tau)\, d\tau  -(s+1)x^{2s} \int_\R \frac{A_{x,\tau}}{\tau }\varphi(t-\tau) \, d\tau.\]
	By substituting these computations, we obtain that indeed $V$, hence $U$ by the definition of $V$, satisfies the equation 
	\[ \pd[t]{ U(x,t) }= \frac{1-2s}{x}\pd[x]{ U(x,t) } +\dpd[x]{ U(x,t) }.\] 
	
	Moreover, using for $x$ large enough the bound
		\[ x^{2s}e^{-\frac{x^2}{4\tau}} \leq M e^{-\frac{1}{4\tau}},\] 
		thanks to the Dominated Convergence Theorem and the limit
			\[\lim_{x\to \infty} x^{2s} e^{-\frac{x^2}{4\tau}} =0,\]
			it yields that \[\lim_{x\to \infty} U(x,t)=0.\]
			
Furthermore, in \eqref{solsts1} by changing variables $\tilde \tau =\displaystyle \frac{\tau}{x^2}$ (but still using $\tau$ as the variable of integration), we have that
		\bgs{ U(x,t)= \frac{1}{4^s\Gamma(s)} \int_0^\infty e^{-\frac{1}{4\tau}} \tau^{-s-1} \phi(t-\tau x^2) \, d\tau.}
		Since $\phi$ is bounded, by the Dominated Convergence Theorem, we have that
		\bgs{\lim_{x\to 0} U(x,t) =\al  \frac{\phi(t)}{4^s\Gamma(s)}  \int_0^\infty e^{-\frac{1}{4\tau}} \tau^{-s-1}  \, d\tau
		=\phi(t),}
		according to \eqref{bla2}. This proves the continuity up to the boundary of the solution $U,$ concluding the proof that the function $U$ defined in \eqref{solsts1} is a continuous solution to the problem \eqref{prob1}.	 
\end{proof}

\section{Relation with the Marchaud fractional derivative}

We  prove here the relation between the parabolic equation studied in Subsection \ref{exuniq} and the Marchaud fractional derivative. Namely, the Marchaud derivative is obtained, in a certain sense, as the trace operator of the extension given by the solution of \eqref{prob1}.

\begin{proof}[Proof of Theorem \ref{teo:mainstat}]
By inserting the expression of $U(x,t)$ from \eqref{solsts2}, we compute 
\bgs{ \lim_{x \to 0^+} x^{-2s} \lr{U(x,t)-\varphi(t)} =\al \lim_{x \to 0^+}x^{-2s}\left( \frac{1}{4^s\Gamma(s)} \int_0^\infty x^{2s}e^{-\frac{x^2}{4\tau}} \tau^{-s-1} \varphi(t-\tau) \, d\tau - \varphi (t)\right).}
Recalling property \eqref{kern} of the kernel, we have that
	\bgs{ \lim_{x \to 0^+} x^{-2s} \lr{U(x,t)-\varphi(t)} =\al \lim_{x \to 0^+} \frac{x^{-2s}}{4^s\Gamma(s)} \int_0^\infty x^{2s}e^{-\frac{x^2}{4\tau}} \tau^{-s-1} \lr{ \varphi(t-\tau) - \varphi (t)} \, d\tau\\
	=\al \lim_{x\to 0^+} \frac{1}{4^s \Gamma(s)} \int_0^\infty e^{-\frac{x^2}{4\tau}}  \frac{ \varphi(t-\tau)-\varphi(t)}{\tau^{s+1}}.}
Now \[ e^{-\frac{x^2}{4\tau}}  \leq 1\] and since $\varphi$ is bounded, we have that 
	\[ \frac{|\varphi(t-\tau)-\varphi(t) |}{\tau^{s+1}} \leq 2M \tau^{-s-1} \in L^1\lr{(1,\infty)}.\]
	On the other hand, recalling that $\varphi$ is locally $C^{\bar{\gamma}}$ we have that
		\[ |\phi(t) - \varphi(t-\tau)|\leq c\tau^{\bar{\gamma}}.\]
		Hence, since $\bar{\gamma}>s$,
			\[    \frac{|\varphi(t-\tau)-\varphi(t) |}{\tau^{s+1}} \leq  c \tau^{\bar{\gamma}-s-1} \in L^1\lr{(0,1)}.\]
			Using the Dominated Converge Theorem, we obtain
\begin{equation}\label{derivative_right}
 \begin{split}\lim_{x \to 0^+} x^{-2s} \left(U(x,t)-\varphi(t)\right) &= \frac{1}{4^s \Gamma(s)} \int_0^\infty \lim_{x\to 0^+} e^{-\frac{x^2}{4\tau}}  \frac{ \varphi(t-\tau)-\varphi(t)}{\tau^{s+1}} \\
				&= \frac{1}{4^s \Gamma(s)} \int_0^\infty\frac{ \varphi(t-\tau)-\varphi(t)}{\tau^{s+1}}.
				\end{split}
				\end{equation}
	Hence for $c_s=4^s\Gamma(s),$ we obtain
		\begin{equation*}
		\begin{split} -  c_s \lim_{x \to 0^+} x^{-2s} \lr{U(x,t)-\varphi(t)} = \int_0^\infty\frac{ \varphi(t)-\varphi(t-\tau) }{\tau^{s+1}}
		= \D^s \varphi(t)
		\end{split}
		\end{equation*}
		by definition \eqref{frader}. This concludes the proof.
\end{proof}

\section{Applications: A Harnack inequality for Marchaud-stationary functions }

In this part of the paper we prove a Harnack inequality for functions that have a vanishing Marchaud derivative in a bounded interval $J$, namely we prove here Theorem \ref{teoHarn}. At this purpose, we use a known Harnack inequality for degenerate parabolic operators, that can be found in \cite{CS}, see Theorem 2.1. For the reader convenience we recall in Proposition \ref{Chiarenza_Serapioni} this result in the case $n=1.$ 
%

 \subsection{Preliminary notions}
We would like to point out that the result given in \cite{CS} was introduced for $n\geq 3$. Nevertheless the same proof works also for $n = 1$ with some adjustments. We recall here the main hypotheses we need to apply the Theorem, adapted in our case $n=1$. It is worth to say that this problem has been studied in a more general fashion in \cite{GuWhee2} and \cite{GuWhee1}.

The degenerate parabolic 
\begin{equation}\label{specific_heat}
w(x)\frac{\partial u}{\partial t}=\pd[x]{}\lr{w(x)\pd[x]{U}},
\end{equation}
is given in $Q=(-R,R)\times (0,T)$, for $R>0$. 

In this particular case, the conductivity coefficient (i.e. the coefficient in front of the $x$ derivative) and the specific heat (the coefficient of the $t$ derivative) coincide. In \cite{CS}, the equation was studied for different coefficients. A more general form of the equation is given in these terms: 
\eqlab{\label{genst1}
w(x)\frac{\partial u}{\partial t}=\pd[x]{}\lr{a(x)\pd[x]{U}}.
}In that case, one requires that
	\[ \lambda^{-1} w(x) \leq a(x)\leq \lambda w(x) \] 
and an integrability condition (also known as a Muckehoupt, or $A_2$ weight condition) on the weight $w$, given by  
\begin{equation}\label{2-weight}
\sup_{J}\lr{\frac{1}{|J|}\int_{J}w(x)\, dx}\, \lr{\frac{1}{|J|}\int_{J}\frac{1}{w (x)}\, dx}=c_0<\infty,
\end{equation}
for any interval $J \subseteq (-R,R)$. The constant $c_0$ is indicated as the $A_2$ constant of $w$. 
In our case, of course, we are left with the condition \eqref{2-weight}.

In addition we consider the functional space
\[ W:=\left\{u\in L^2(0,T;H^{1}_0(J,w)) \mbox{ s.t. }\quad \frac{\partial u}{\partial t}\in  L^2(0,T;L^2(J,w))\right\}.\]
We denote here by $L^2(J,w),$ the Banach space of measurable functions $u$ with finite weighted norm
$$
\|u\|_{p,w;J}=\lr{\int_{J}|u|^2 w \,dx}^{1/2 }<\infty,
$$
by $H^1(J,w)$ the completion of $C^{\infty}(\overline{J})$ under the norm
$$
\|u\|_{1,w;J}=\lr{\int_{J}(u^2+|\partial_x u|^2)w \, dx}^{1/2}
$$
and by $H^1_0 (J,w)$ the completion of $C^{\infty}_0$ under the norm
$$
\|u\|_{1,w;J}=\lr{\int_{J}|\partial_x u|^2w \,dx}^{1/2}.
$$ The time dependent Sobolev space $L^2\lr{0,T; H_0^1(J,w)}$ is defined as the set of all measurable functions $ u$ such that
 \[ \|u\|_{L^2\lr{0,T; H_0^1(J,w)}}  := \lr{ \iint_{(0,T)\times J} |u(x,t)|^2 w(x)\, dx \, dt }^{\frac{1}2}< \infty.\] 

In this setting, we introduce the notion of weak solution of problem \eqref{specific_heat}.
\begin{defn}
We say that $u\in L^2(0,T;H^1(J,w))$ is a weak solution of \eqref{specific_heat} in $J \times (0,T)$ if, for every $\eta\in W,$ such that $\eta(0)=\eta(T)=0$ we have that
\begin{equation}\label{weak_solution}
\iint_{J\times (0,T)} w(x) \left(\pd[x] {u} \pd[x] {\eta} - u\pd[t]{\eta}\right)\, dx\, dt=0.
\end{equation}
\end{defn}

We have the next proposition (see Theorem $2.1$ in \cite{CS}). 
\begin{prop}\label{Chiarenza_Serapioni}
Let $u$ be a positive solution in $(-R,R)\times (0,T)$ of (\ref{specific_heat}) and assume that condition \eqref{2-weight} holds, with constant $c_0$. Then there exists $\gamma=\gamma(c_0)>0$ such that
\begin{equation}
\sup_{\lr{t_0-\frac{3\rho^2}4 , t_0-\frac{\rho^2}{4}}\times-\lr{\frac{\rho}2,\frac{\rho}2}} u\leq \gamma \inf_{\lr{t_0+\frac{3\rho^2}4 , t_0+\rho^2}\times-\lr{\frac{\rho}2,\frac{\rho}2}}u
\end{equation} \label{paraharnack}
holds for $t_0\in (0,T)$ and any $\rho$ such that $0<\rho<R/2$ and $[t_0-\rho^2,t_0+\rho^2]\subset (0,T)$.

\end{prop}
\begin{oss}
The reader can easily imagine the general situation in any dimension as explicated in Theorem 1.2 in \cite{CS}, where the coefficient $a(x)$ in \eqref{genst1} is a matrix and the domains are cylinders.  
We have stated the Harnack inequality in $ (0,T).$ Nevertheless with a change of coordinates in space and time, we can always say that the Harnack inequality holds in any subset of $(R_1,R_2)\times(\tau_1,\tau_2),$ where $R_1,R_2,\tau_1,\tau_2\in \mathbb{R}.$ 
\end{oss}
\subsection{Reflection of the solution}
We consider here that $\D^s \phi (t)=0$ in an interval $J$. By taking the reflection $\tilde U$ of the solution of problem \eqref{prob1}, we prove that $\tilde U$ is a solution in a weak sense of \eqref{prob1} across $x=0$.
 
It is useful to introduce a weak version of the limit $\displaystyle \lim_{x\to 0} x^{1-2s} \partial_x u(x,t)$. In this sense, we have:
\begin{defn}
We say that in a weak sense
\bgs{  \lim_{x\to 0} x^{1-2s}\pd[ x ]{U}(x,t) =0}
if and only if, for any  $\eta\in W,$ such that $\eta(0)=\eta(T)=0$ we have that
\eqlab{\label{weak_limit} \lim_{x\to 0} \int_0^T x^{1-2s} \pd[x]{ U} \,\eta\,dt =0.}
\end{defn}
\begin{lem}\label{tildeu}
Let $U\colon \R \times [0,\infty)\to \R$ be a solution of the problem (\ref{prob1}) such that, in a weak sense
$  \displaystyle\lim_{x\to 0^+} x^{1-2s} \partial_x U(x,t)=0.$
Then the extension 
\syslab [\tilde U(x,t):=] { & U(x,t), & (x,t)\in &[0,+\infty)\times (0,T)\\
&U(-x,t), & (x,t)\in& (-\infty,0)\times (0,T)}
is a weak solution of
 \eqlab{\label{weaksol}
\frac{\partial (|x|^{1-2s}U(x,t))}{\partial t} = \pd[x]{}\lr{|x|^{1-2s}\pd[x]{U(x,t)}}.
}
in $(-R,R)\times (0,T)$.
\end{lem}
\begin{proof}
We claim that the extension $\tilde U$ is a weak solution of \eqref{weaksol}, hence that 
	\eqlab{\label{bla22} \int_{(-R,R)\times(0,T)} |x|^{1-2s}\lr{ \pd[x]{ \tilde U} \pd[x]{ \eta} - \tilde U \pd[t]{ \eta}}\, dx\, dt =0. }
	We compute, integrating by parts
	\bgs{  \int_0^T  \al \lr{  \int_0^R x^{1-2s} \pd[x]{ \tilde U} \, \pd[x]{ \eta} \, dx}\,dt \\ = \al  \int_0^T R^{1-2s} \pd[x]{ \tilde U} (R,t) \, \eta(R,t)  \, dt - \lim_{x\to 0} \int_0^T x^{1-2s}\pd[x]{ U} \,  \eta \, dt \\ \al - \int_0^T \lr{ \int_0^R \pd[x]{} \lr{ x^{1-2s} \pd[x] { \tilde U} } \eta \, dx }\, dt \\
	= \al \int_0^T R^{1-2s} \pd[x]{ \tilde U} (R,t)   \eta (R,t) \, dt  -\int_0^T\lr{ \int_0^R x^{1-2s}  \pd[t]{ \tilde U} \, \eta\, dx}\, dt ,}
	where we have used the weak limit in \eqref{weak_limit} and the fact that $\tilde U$ solves equation \eqref{weaksol}.
In the same way, one obtains that
 	\bgs{  \int_0^T  \al \lr{  \int_{-R}^0 (-x)^{1-2s} \pd[x]{ \tilde U} \, \pd[x]{ \eta} \, dx}\,dt =  \int_0^T R^{1-2s} \pd[x]{ \tilde U}  (-R,t)  \eta (-R,t)  \, dt  \\ \al -\int_0^T \lr{ \int_{-R}^0 (-x)^{1-2s}  \pd[t]{ \tilde U} \, \eta \, dx}\, dt ,}
 	therefore, by summing up, 
 	\bgs{ \int_{(-R,R)\times(0,T)} \al |x|^{1-2s}\pd[x] {\tilde U} \pd[x]{ \eta} \, dx \, dt  \\ = \al \int_0^T R^{1-2s} \lr{\pd[x]{ \tilde U} (R,t)  \eta(R,t)   - \pd[x]{ \tilde U}(-R,t) \eta (-R,t)} \, dt \\ \al -\int_0^T \lr{\int_{-R}^R |x|^{1-2s} \pd[t]{ \tilde U} \, \eta \, dx} \, dt.}
Hence
\bgs{ 	\int_{(-R,R)\times(0,T)} \al     |x|^{1-2s} \lr{ \pd[x]{\tilde U} \pd[x]{ \eta} - \tilde U \pd[t]{ \eta}}\, dx\, dt  \\ = \al    \int_0^T R^{1-2s}\lr{ \pd[x]{ \tilde U} (R,t) \eta(R,t)  - \pd[x]{ \tilde U}(-R,t)  \eta (-R,t) } \, dt  \\ \al   - \int_0^T\lr{ \int_{-R}^R |x|^{1-2s} \lr{\pd[t]{ \tilde U} \, \eta - \tilde U \pd[t]{ \eta}} \, dx}\, dt \\
	= \al \int_0^T R^{1-2s} \lr{ \pd[x]{ \tilde U} (R,t) \eta(R,t)  - \pd[x]{ \tilde U}(-R,t)  \eta (-R,t)  } \, dt \\ 
	\al-   \int_{-R}^R|x|^{1-2s} 	\lr{ \tilde U (x,T)  \eta (x,T)  - \tilde U(x,0)\eta(x,0) } \, dx\\
	=\al 0,}
since $\eta(x,T)=\eta(x,0)=0$ and $\eta(R,t)=\eta(-R,t)=0$. This is the claim in \eqref{bla2}, and we conclude the proof of the Lemma.
\end{proof}
\subsection{The Harnack inequality for the Marchaud derivative} 

We show here that the Harnack inequality for the Marchaud derivative can be deduced from the Harnack inequality applied to the extension operator.
\begin{proof}[Proof of Theorem \ref{teoHarn}]
We consider $U$ to be the extension of $\phi$, as introduced in Theorem \ref{teo:mainstat}. Since $\phi$ is nonnegative, given the explicit solution $U$ in \eqref{solsts1}, the function $U$ is also positive. Now, we reflect $U$ and obtain $\tilde U$, as we have done in Lemma \ref{tildeu}. 

We take at first $J=(0,T)$. Since $\D^s \phi (t)=0$ in $(0,T)$, we have by definition that
	\[ \lim_{x\to 0} x^{-2s} \pd[x]{ U(x,t)}=0,\]
	and thanks to Lemma \ref{tildeu}, we obtain that $\tilde U$ is a weak solution of \eqref{weaksol} in, say, $(-R,R)\times(0,T)$ for a given $R>0$. Moreover, the function $|x|^{1-2s}$ satisfies the condition \eqref{2-weight}, and according to Proposition \ref{Chiarenza_Serapioni}, we have that
	\[ \sup_{\lr{t_0-\frac{3\rho}4 , t_0-\frac{\rho}{4}}\times-\lr{-\frac{\rho}2,\frac{\rho}2}} u\leq \gamma \inf_{\lr{t_0+\frac{3\rho}4 , t_0+1}\times-\lr{\frac{\rho}2,\frac{\rho}2}}u\] for any $\rho$ such that $0<\rho<R/2$ and $[t_0-\rho^2,t_0+\rho^2]\subset (0,T)$.
	It suffices now to slice the domain at $x=0$ to obtain that
		\[\sup_{\lr{t_0-\frac{3\rho^2}4 , t_0-\frac{\rho^2}{4}}} u(0,t) \leq \gamma \inf_{\lr{t_0+\frac{3\rho^2}4 , t_0+\rho^2}}u (0,t) ,\]
		hence
		\[ \sup_{\lr{t_0-\frac{3\rho^2}4 , t_0-\frac{\rho^2}{4}}}\phi(t)\leq \gamma \inf_{\lr{t_0+\frac{3\rho^2}4 , t_0+\rho^2} }\phi(t) \]
		for any $\rho$ such that $0<\rho<R/2$ and $[t_0-\rho^2,t_0+\rho^2]\subset (0,T)$.
		
		Now, in order to prove that the Harnack inequality holds on any interval $J \subset \R$, one considers a translation of $U$, namely for any $\theta\in \R$, the function $U_\theta (x,t):=U(x,t+\theta)$, and reflects it as Lemma \ref{tildeu}. Then $\tilde U_\theta$ is a weak solution of \eqref{weaksol}, and $\tilde U_\theta(0,t)=\phi(t+\theta)$.
		One obtains then, as a consequence of the Harnack inequality for the solution $U_\theta$, the following
\[ \sup_{\lr{t_0-\frac{3\rho^2}4 , t_0-\frac{\rho^2}{4}}}\phi(t+\theta)\leq \gamma \inf_{\lr{t_0+\frac{3\rho^2}4 , t_0+\rho^2} }\phi(t+\theta) \]
for any $\rho$ such that $0<\rho<R/2$ and $[t_0-\rho^2,t_0+\rho^2]\subset (0,T)$. Therefore
	\[ \sup_{\lr{t_0-\theta-\frac{3\rho^2}4 , t_0-\theta -\frac{\rho^2}{4}}}\phi(t)\leq \gamma \inf_{\lr{t_0-\theta+\frac{3\rho^2}4 , t_0-\theta+\rho^2} }\phi(t) \]
for any $\rho$ such that $0<\rho<R/2$ such that $[t_0-\theta-\rho^2,t_0-\theta+\rho^2]\subset (0,T)$. As $\theta$ is arbitrary, one concludes that
	\[ \sup_{\lr{t_0-\frac{3\rho^2}4 , t_0 -\frac{\rho^2}{4}}}\phi(t)\leq \gamma \inf_{\lr{t_0+\frac{3\rho^2}4 , t_0+\rho^2} }\phi(t) \]
	for any $\rho$ such that $0<\rho<R/2$ and $[t_0-\rho^2, t_0+\rho^2]\subset J$. 
		This concludes the proof of the Theorem.
\end{proof}
\begin{oss}\label{Holder_regularity_ine}

We would like to point out that the Harnack type inequality obtained in Theorem \ref{teoHarn} 
 can be equivalently stated as follows.
 Let us define for every $\delta>0$ and for every $\tau\in \mathbb{R}$ the sets
\bgs{ I(\tau,\delta)=& [\tau-\frac{15}{8}\delta,\tau+\frac{1}{8}\delta],\\
I^+(\tau,\delta)= &[\tau-\frac{15}{8}\delta,\tau-\frac{7}{4}\delta]\\
I^-(\tau,\delta)=& [\tau-\frac{1}{8}\delta,\tau+\frac{1}{8}\delta].}

With this notation, the Harnack inequality affirms that for every  $I(\tau,\delta)\subset J$
$$
\sup_{I^+(\tau,\delta)} \varphi\leq \gamma \inf_{I^-(\tau,\delta)}\varphi.
$$
\end{oss}

\section{Backward equation}\label{rightderiv}
We now consider the case of the right Marchaud fractional derivative, denoted by ${\D}^s_-\phi.$ The following result is true:
\begin{teo} \label{teo:mainstat_back}
Let $s\in (0,1) $ and $\varphi \colon \R \to \R$ be a bounded, locally $C^{\bar{\gamma}}$ function for $s<\bar{\gamma}\leq1$. Let $U_-\colon [0,\infty)\times \R\to \R$ be 
a solution of the problem 
\syslab{ \label{prob2}
	& -\pd[t]{ U(x,t)} = \frac{1-2s}x \pd[x]{ U(x,t)} + \dpd[x]{ U(x,t)},  & &   (x,t)\in(0,\infty)\times \R\\
	 &U(0,t)=\varphi(t),  &&t\in\R\\
	  & \lim_{x \to \infty} U(x,t)=0. && 
}

Then $U_-$ defines the extension operator for $\phi$, such that
\eqlab{\label{mainstat2} \D^s_- \varphi(t)=-\lim_{x\to 0^+} c_s x^{-2s}(U_-(x,t)- \varphi(t)) ,}
where \[ c_s= 4^s\Gamma(s).\] \end{teo}

We do not repeat all the computations, that are very similar to the case of the left Marchaud-derivative ${\D}^s\equiv {\D}^s_+.$

We only point out that if $U_-$ is a solution of \eqref{prob2}, then $U_-(x,t)=U(x,-t),$ where $U$ is the solution of the differential equation in \eqref{prob1}.
 Thanks to Theorem \ref{teo:sol} and keeping in mind (\ref{solsts2}), we get

\eqlab{\label{solsts2_back} U_-(x,t)=\frac{1}{4^s \Gamma(s)} x^{2s} \int_0^{\infty} e^{-\frac{x^2}{4\tau}} \tau^{-s-1}  \varphi(t+\tau) \, d\tau.}
Recalling the computations in (\ref{derivative_right}) and the properties of the kernel $\Psi_s$, see formula (\ref{kencalc1}),  we obtain that

\begin{equation}
\begin{split}
\lim_{x \to 0^+} x^{-2s} \left(U_-(x,t)-\varphi(t)\right) &= \lim_{x \to 0^+} \frac{x^{-2s}}{4^s\Gamma(s)} \int_0^\infty x^{2s}e^{-\frac{x^2}{4\tau}} \tau^{-s-1} \lr{ \varphi(t+\tau) - \varphi (t)} \, d\tau\\
	&= \lim_{x\to 0^+} \frac{1}{4^s \Gamma(s)} \int_0^\infty e^{-\frac{x^2}{4\tau}}  \frac{ \varphi(t+\tau)-\varphi(t)}{\tau^{s+1}}.
\end{split}
\end{equation}	
Thus, using the same argument as in the proof of Theorem \ref{teo:mainstat}, we conclude that
\begin{equation}
\begin{split}
\lim_{x \to 0^+} x^{-2s} \left(U_-(x,t)-\varphi(t)\right) &= \frac{1}{4^s \Gamma(s)} \int_0^\infty \frac{ \varphi(t+\tau)-\varphi(t)}{\tau^{s+1}},
\end{split}
\end{equation}	
that is 
$$
{\D}_-^s\varphi(t)=-c_s\lim_{x \to 0^+} x^{-2s} \left(U_-(x,t)-\varphi(t)\right). 
$$
It is worth to say that ${\D}_-^{1-s}{\D}_-^s\varphi(t)=-\displaystyle \frac{d\varphi}{dt}.$ Hence, using a different notation we can write that
$$
{\D}_+^s\varphi(t)=\lr{\frac{d}{dt}}^s\varphi,\quad {\D}_-^s\varphi(t)=\lr{-\frac{d}{dt}}^s\varphi.
$$

\section{Appendix}\label{intfrd} 

In the Appendix, we provide some details on the Marchaud derivative.

  First of all, as stated in the Introduction, the Marchaud fractional operator $\D^s \varphi$ is well defined for $\phi$ bounded, locally $C^{\bar{\gamma}}$ for $\bar{\gamma}>s$. Indeed, we have that:
	\begin{equation*}
	\begin{split}
	\int_0^\infty \frac{\varphi(t)-\varphi(t-\tau)}{\tau^{s+1}}\, d\tau &=   \int_{1}^\infty \frac{\varphi(t)-\varphi(t-\tau)}{\tau^{s+1}}\, d\tau + \int_{0}^1  \frac{\varphi(t)-\varphi(t-\tau)}{\tau^{s+1}}\, d\tau . \\
	&=  I_1+I_2.
	\end{split}
	\end{equation*}
	Since $\phi$ is bounded, we have 
	\begin{equation}
	\begin{split}
	 I_1 \leq 2\|\varphi\|_{L^{\infty}(\R)} \int_1^\infty  \frac{1}{\tau^{s+1}}\, d\tau 
				=   C_{s,\varphi}.
				\end{split}
				\end{equation}
				Moreover, $\varphi$ is locally H{ö}lder, hence in $(0,1)$ we may write
					\[|\varphi (t)- \varphi(t- \tau)| \leq c {\tau^{\bar{\gamma}}}.\]
					Therefore
	\bgs{ I_2\leq  c  \int_0^1  \tau^{\bar{\gamma}-s-1} \, d\tau
			\leq   C_{s,\bar{\gamma}},}
			recalling that $\bar{\gamma}>s$.

There are, in literature, many other definitions of fractional derivatives. The interested reader can consult, for instance, \cite{MillerRoss} or \cite{samkokilbas}  for further details. Here, we recall only the Riemann-Liouville fractional derivative, defined as

$$
\mathcal D^{s}_{\pm} f(t)=\frac{\pm 1}{\Gamma(1-s)}\frac{d}{dt}\int_{0}^{\infty}\frac{f(t\mp\tau)}{\tau^{s}}d\tau
$$
for $s\in \mathbb{C},$ $0<\Re s<1,$ see \cite{Samko}, Definition 1.16 and the Caputo derivative (see formulas 2.4.17 and 2.4.18 in \cite{KST06}), given by
\bgs{	&{D}^s_{\pm} f(t):= \displaystyle \frac{\pm 1}{\Gamma(1-s)}\int_0^\infty \frac   {f'(t\mp \tau)}{\tau^s}\, d\tau .}
The definitions of Caputo or Riemann-Liouville are related to the Marchaud definition. Indeed, as one can see in formula (13.2) in the monograph  \cite{samkokilbas}, the Marchaud derivative is an extension of Riemann-Liouville's, with weaker conditions on the function $f$. For sufficiently smooth $f$ (say absolutely continuous, for instance), integrating by parts in the Riemann-Liouville definition, one can deduce the Marchaud notion  (see also Theorem 1.17 in \cite{Samko}).
 
As a further remark, the Marchaud derivative coincides with the notion of fractional derivative given by  Gr\"unwald and Letnikov, see \cite{Grunwald}, and Theorem 20.4 in \cite{samkokilbas} for the proof.

We  like also to remark that, 
just adapting the constant $c_s$ given in Theorem \ref{teo:mainstat} by fixing $c_s=\frac{4^s\Gamma(s)s}{\Gamma(1-s)},$ in (\ref{mainstat1}) we straightforwardly obtain the definition (\ref{mdefct}). The advantage of this choice is that $\D_{\pm}^s\varphi\to\varphi$ as $s\to 0^+$ and $D_{\pm}^s\varphi\to \varphi'$ as $s\to 1^-.$ Indeed, it is well known that the Marchaud derivative is not defined for $s=0$ and $s=1$ because in those cases the integral term in (\ref{mdefct}) (as in (\ref{frader})) does not converge. However, one is able to pass to the limit by using the constant term from definition \eqref{mdefct}, that in this sense plays a fundamental role.

\bibliography{CF_ext_final}
\bibliographystyle{plain}

\end{document}